\documentclass[12pt]{amsart}
\usepackage[mathscr]{eucal}
\newtheorem{thm}{Theorem}
\newtheorem*{thm*}{Theorem}
\newtheorem{lemma}[thm]{Lemma}
\begin{document}
\newcommand{\R}{{\mathbb R}}
\newcommand{\C}{{\mathbb C}}
\newcommand{\Z}{{\mathbb Z}}
\newcommand{\Id}{{\mathbb I}}
\newcommand{\Ker}{\mathsf{Ker}}
\newcommand{\Aut}{\mathsf{Aut}}
\newcommand{\Ad}{\mathsf{Ad}}
\newcommand{\Inn}{\mathsf{Inn}}
\newcommand{\tr}{\mathsf{tr}}
\newcommand{\SL}[1]{{\mathsf{SL}}({#1})}
\newcommand{\slt}{{\SL{2}}}
\newcommand{\pslt}{{\SL{2}}}
\newcommand{\GL}[1]{{\mathsf{GL}}({#1})}
\newcommand{\PGL}[1]{{\mathsf{PGL}}({#1})}
\newcommand{\PSL}[1]{{\mathsf{PGL}}({#1})}
\newcommand{\pgl}{{\PGL{2,\Z}}}
\newcommand{\hmg}{\mathsf{Hom}(\pi ,G)}
\newcommand{\Q}{{\mathcal Q}_a}
\newcommand{\ti}{\tilde\iota}
\newcommand{\ha}{\tilde A}
\newcommand{\hb}{\tilde B}
\newcommand{\hc}{\tilde C}
\newcommand{\hd}{\tilde D}
\newcommand{\mC}{\mathcal{C}}
\newcommand{\mL}{\mathcal{L}}
\newcommand{\bmC}{\overline{\mC}}
\renewcommand{\P}{\mathbb{P}}
\renewcommand{\div}{{\mathsf{div}}}
\newcommand{\Sof}{{\mathbb{S}_{0,4}}}

\title[Character varieties and cubic surfaces]
{Affine cubic surfaces and relative $SL(2)$-character varieties 
of compact surfaces
}
\author[Goldman]{William M.~Goldman}
\address{ Mathematics Department,
University of Maryland, College Park, MD  20742 USA  }
\email{ wmg@math.umd.edu }
\author[Toledo]{Domingo ~Toledo}
\address{ Mathematics Department,
University of Utah, Salt Lake City, Utah 84112  USA  }
\email{ toledo@math.utah.edu }

\thanks{Goldman gratefully acknowledges partial support from 
National Science Foundation grant
DMS-070781.
Toledo gratefully acknowledge partial support from 
National Science Foundation grant DMS-0600816.}
\date{\today}
\begin{abstract}
A natural family of affine cubic surfaces arises from
$\slt$-characters of the 4-holed sphere and the 1-holed torus.
The ideal locus is a tritangent plane which is generic in the
sense that the cubic curve at infinity consists of three lines
pairwise intersecting in three double points.
We show that every affine cubic surface which is smooth at infinity 
and whose ideal locus is a generic tritangent plane
arises as a relative $\slt$-character variety of the
4-holed sphere. Every such affine cubic for which all the periodic
automorphisms of the tritangent plane extend to automorphisms of the
cubic arises as a relative $\slt$-character variety of a 1-holed torus.
\end{abstract}
\maketitle
\section*{Introduction}
Various moduli problems for surface group representations in $\slt$ 
lead to complex cubic surfaces in affine 3-space $\C^3$.  There are two classical examples of the following situation.  Let $\Sigma$ be a surface with non-empty boundary.  For each connected component of the boundary fix a conjugacy class in $\slt$.   Consider the moduli space of flat $SL(2)$- bundles over $\Sigma$ whose holonomy on each component belongs to the fixed conjugacy class.   Then this moduli space is homeomorphic to an  affine cubic surfaces of the form
\begin{equation*}
x^2 + y^2 + z^2 + x y z = f(x,y,z) 
\end{equation*}
where $f(x,y,z)$ is a polynomial of degree 1 depending on the fixed conjugacy classes.  

In a similar vein, we  can restrict to
flat $\slt$-bundles 
whose  holonomy restricted  to the boundary has fixed trace.
In turn, these correspond to $\slt$-representations of $\pi_1(\Sigma)$
whose restriction to the boundary components are constrained to have fixed traces.

The cases of interest occur when $\Sigma$ is homeomorphic to either a
1-holed torus or a 4-holed sphere. When $\Sigma$ is a 1-holed torus,
then $f(x,y,z)$ is a constant, which equals the boundary trace minus $2$.  
When $\Sigma$ is a 4-holed sphere, and $a,b,c,d$ are the traces of the restriction to the four boundary components,  then
\begin{align*}
f(x,y,z) & = (a b + c d) x +  (b c + d a) y + (c a + b d) z \\
& \quad + (4 - a^2 - b^2 - c^2 - d^2 - a b c d) 
\end{align*}
  See (9) in  p.\ 298 of \cite{FrickeKlein} for this formula, and  (7) in  p.\ 301 for the one-holed torus.  
This family of cubic surfaces appears in many different contexts.  In addition to works cited below,  it also appears  in Oblomkov's work   (see Theorem 2.1 of \cite{Oblomkov}) and in recent work of  Gross, Hacking and Keel (see Ex.\ 5.12 of \cite{Gross}).

For fixed values of the boundary traces, the cubic surfaces that occur are all of the form
\begin{equation}
\label{eq:affinecubic}
x^2 + y^2 + z^2 + x y z = p x + q y + r z + s.
\end{equation}
Since the isomorphism classes of cubic surfaces depend on four parameters, it is natural to ask if all cubic surfaces occur in this way. 

This  paper has two purposes.   The first is to prove that every cubic surface of this
form (\ref{eq:affinecubic}) arises from a representation of the fundamental group of the
$4$-holed sphere in $\mathsf{SL}(2,\C)$, see Theorem \ref{thm:onto}.   This theorem may have been known classically.   
Versions in the real domain may be found in  Fricke-Klein. 
More recent statements may be found in 
Boalch~\cite{Boalch}, Cantat-Loray~\cite{CantatLoray}, and
Iwasaki~\cite{Iwasaki}.  Nevertheless we are not aware of any published proof of this result, and the proof we present here is elementary and direct.

The second purpose is to characterize cubic surfaces 
defined by (\ref{eq:affinecubic}).  
These are the affine cubic surfaces that are non-singular at infinity and that intersect the hyperplane at infinity in a {\em generic tritangent plane}.     We recall some terminology:  a {\em tritangent plane} to a cubic surface is a plane that intersects the surface in three lines.  A {\em generic tritangent plane} is a tritangent plane where the three lines are in general position.  
If a tritangent plane is not generic, then the three lines meet at a point, 
called an {\em  Eckardt point} of the cubic surface.    
We will see the classical fact that all non-singular cubic surfaces contain such a generic tritangent plane, 
so the family (\ref{eq:affinecubic}) contains 
representatives of every non-singular projective cubic surface.  

It is also classical that a singular cubic surface satisfies our two conditions if and only if its singularities are of certain types that we list below.  From this we obtain a complete list of the possible singularities of the relative character varieties, namely $A_1$, $2A_1$, $3A_1$, $4A_1$, $A_2$, $A_3$ and $D_4$. In a sequel we plan to
discuss the classification of these singularities and their interpretation
in terms of representations of the fundamental group of $\Sigma$.

\section*{Acknowledgements}
We wish to thank Philip Boalch, 
Serge Cantat, Lawrence Ein, Frank Loray and Walter Neumann for
helpful discussions.  We thank Mladen Bestvina for supplying us with a proof of the properness of the trace map, thereby greatly simplifying the first version of this paper.

\section{The Family of Cubics Surfaces}

The affine cubic surfaces of the form (\ref{eq:affinecubic})  can be characterized as follows.  The
ideal locus (intersection of the projective completion with the
hyperplane at infinity) consists of smooth points. The ideal
hyperplane is a tritangent plane which meets the surface in three
lines in general position (a {\em generic tritangent plane.)\/} Recall 
(see \cite{Segre} for details)  
that a {\em tritangent plane\/} to a cubic surface $S\subset \P^3$ is
a plane $P$ that intersect $S$ in three lines.  We say that $S\cap P$
is {\em generic\/} if $S\cap P$ consists of three distinct lines,
pairwise intersecting in three points.  The plane $P$ is then tangent
to $S$ at the three points of intersection, hence the name
tritangent plane.  If $S$ is singular, we require, in addition, 
that $P\cap S$ consists entirely of non-singular points of $S$.

A smooth cubic surface has $45$ tritangent planes.
Each tritangent plane of a generic cubic surface is generic in the
above sense.  If $P$ is not generic, then the three lines of
intersection of $P$ with $S$ go through a single point, called an {\em
Eckardt point\/} of $S$.  Let us define an Eckardt point of a smooth cubic surface $S$  to  be a point $Q\in S$ where three lines in $S$ intersect.  Then these three lines must be coplanar, since they are all tangent to $S$ at $Q$, and there can be no other lines in $S$ passing through $Q$, since otherwise we would have a plane intersecting $S$ in a curve of degree larger than $3$.  In this way, for any smooth cubic surface $S$, we get a one  to one correspondence between non-generic tritangent planes and Eckardt points.

It is classically known that the maximum number of Eckardt points is $18$, achieved exactly by the 
{\em Fermat cubic\/}
\begin{equation*}
X^3 + Y^3 + Z^3 + W^3 = 0.
\end{equation*}
See  p.152 (end of \S 100) of \cite{Segre} for a list of the surfaces with Eckardt points and the number of such points,  namely  1,2,3,4,6,9,10,18.   

Thus the generic smooth cubic surface $S$ has no Eckardt
points, and most  of the $45$ tritangent planes to any $S$ contain no
Eckardt points.  In fact, for any surface, at least $27$ of its tritangent planes are generic.   

From a different point of view, the collection of smooth surfaces with Eckardt points forms a divisor in the moduli space of smooth cubic surfaces.    This divisor is totally geodesic in the complex hyperbolic structure on this space described in \cite{ACT}.   See \S 11 of \cite{ACT} for a proof that the surfaces with Eckardt points form a totally geodesic divisor.

With this information in mind, let us return to the characterization of surfaces of the form (\ref{eq:affinecubic}).    Pick a surface $S$ with a generic tritangent plane $P$.  In a suitable homogeneous coordinate system $(X,Y,Z,W)$ for $\P^3$,
$P$ is given by the equation $W=0$
and the intersection of $S$ and $P$ described in by the equations
\begin{equation*}
X Y Z = 0, \ W =0.
\end{equation*}

An affine cubic surface whose ideal locus is a generic tritangent plane 
is therefore defined by a cubic polynomial of the form 
\begin{equation*}
x y z +  f(x,y,z) = 0
\end{equation*}
where $f(x,y,z)$ is polynomial of degree $\le 2$.
Writing 
\begin{align*}
f(x,y,z) &\; = \;
f_{11} x^2  + f_{12} xy  + f_{22} y^2 \\
&\quad  +f_{13} x z  +f_{23} yz  +f_{33} z^2   
\\ & \qquad + p x + q y + r z + s,
\end{align*}
applying a translation 
\begin{equation*}
\bmatrix x \\ y \\ z  \endbmatrix \longmapsto 
\bmatrix x - f_{23} \\ y - f_{13} \\ z-f_{12}  \endbmatrix 
\end{equation*}
eliminates the cross term and we may assume: 
\begin{equation*}
f_{12}=f_{13}=f_{23} = 0.
\end{equation*}
Furthermore if any of $f_{11},f_{22},f_{33}$ vanish, then $S$ is
singular at infinity.  Finally by applying the diagonal linear
transformation with  entries $(f_{22}f_{33})^\frac{1}{2}$, $(f_{11}f_{33})^\frac{1}{2}$, $(f_{11}f_{12})^\frac{1}{2}$, we may assume that
\begin{equation*}
f_{11}=f_{22}=f_{33}= 1. 
\end{equation*}
Hence we obtain the normal form
\begin{equation*}
x^2 + y^2 + z^2 + x y z = p x + q y + r z + s.
\end{equation*}
Thus we have proved the following theorem:

\begin{thm}
\label{thm:afinnesurfaces}
A projective cubic surface is projectively equivalent to a surface with Equation (\ref{eq:affinecubic}) if and only if it has a generic tritangent plane.   In particular, every smooth cubic surface is projectively equivalent to a surface with Equation (\ref{eq:affinecubic}) for suitable $p,q,r,s$.
\end{thm}


\section{Surjectivity of the trace map}

The relative character variety of the $4$-holed sphere 
$\Sof$ is the restriction of the projection
\begin{align*}
V & \longrightarrow \C^4 \\
\bmatrix a \\ b \\ c \\ d \\ x \\ y \\z
\endbmatrix & \longmapsto \bmatrix a \\ b \\ c \\ d \endbmatrix
\end{align*}
to the hypersurface $V\subset\C^7$ defined by
\begin{align}
\label{eq:cub4holedSp}
x^2 + y^2 + z^2 + x y z  & =\; p(a,b,c,d)\, x \\ 
& \quad +\; q(a,b,c,d) \,y \notag\\
& \qquad  + \;r(a,b,c,d) \,z \notag \\
& \qquad \quad  +\; s(a,b,c,d) \notag 
\end{align}
where
\begin{align}\label {eq:linearcoeffs}
p(a,b,c,d) & = a b + c d  \\ 
q(a,b,c,d) & = b c + d a  \notag\\ 
r(a,b,c,d) & = c a + b d \notag\\ 
s(a,b,c,d) & = 4 - a^2 - b^2 - c^2 - d^2 - a b c d \notag
\end{align}
We show that every  affine cubic of the form 
\eqref{eq:cub4holedSp}
arises as a relative $\slt$-character variety of a $4$-holed sphere
for some choice of boundary traces $(a,b,c,d)$ That is,
\begin{thm}\label{thm:onto}
The mapping
\begin{align*}
\C^4 & \xrightarrow{\Phi}  \C^4 \\
\bmatrix a \\ b \\ c \\ d \endbmatrix &\longmapsto
\bmatrix 
p(a,b,c,d) \\q(a,b,c,d) \\r(a,b,c,d) \\s(a,b,c,d) \endbmatrix 
\end{align*}
is onto.
\end{thm}
\noindent
Theorem~\ref{thm:onto} characterizes affine cubic surfaces arising from
the four-holed sphere. Here is a characterization of those affine
cubic surfaces arising from the one-holed torus.

\begin{thm}\label{thm:torus}
Let $V_{p,q,r,s}$ denote the affine cubic defined by \eqref{eq:cub4holedSp}.
If the finite group of automorphisms
\begin{equation*}
\Delta \; := \; 
\bigg\{ \bmatrix \pm 1 & 0 & 0 \\ 0 & \pm 1 & 0 \\ 0 & 0 & \pm 1
\endbmatrix \bigg\} \,\bigcap\, \mathsf{SL}(3)
\end{equation*}
preserves $V_{p,q,r,s}$, then $p = q = r = 0$. In that case $V_{0,0,0,s}$ 
corresponds to a relative $\slt$-character variety of a 1-holed torus with 
boundary trace $s + 2$.
\end{thm}


\begin{lemma}\label{lem:proper}
The map $\Phi$ is proper.
\end{lemma}
\noindent
The proof of 
Lemma~\ref{lem:proper}
is based on the following observations:
Differences of linear coefficients $p,q,r$ factor as follows:
\begin{align}\label{eq:differences}
p - q & \,=\, (a - c)(b - d) \notag \\ 
q - r & \,=\, (b - a)(c - d)  \\ 
r - p & \,=\, (a - d)(b - c). \notag
\end{align}
Sums of linear coefficients $p,q,r$ likewise factor:
\begin{align}\label{eq:sums}
p + q & \,=\, (a + c)(b + d) \notag \\ 
q + r & \,=\, (b + a)(c + d)  \\ 
r + p & \,=\, (a + d)(b + c) \notag
\end{align}

\begin{lemma}\label{lem:three} 
Suppose $p(a,b,c,d),\ q(a,b,c,d)$ and $ r(a,b,c,d)$ are bounded.
Then at least three of $a,b,c,d$ are within bounded distance of each other.
\end{lemma} 
\begin{proof}
Suppose $|p|, \ |q|, \ |r|\le C$ for some positive constant $C$.
Then the first equation in  \eqref{eq:differences} implies  
\begin{equation*}
|a -c||b-d|\le 2C.
\end{equation*} 
Thus at least one of the factors is $\le C'$, where $C' = \sqrt{2C}$.   Suppose $|c-a|\le C'$.  Then the second equation in 
\eqref{eq:differences} implies $|b -a |$ or $|c - d|$ is $\le C'$.   
Suppose $|b -a| \le C'$.  
By the triangle inequality,
\begin{equation*}
|b - c| \le C'' = 2C'.
\end{equation*}   
Thus the three points $a,b,c$ are within distance $C''$ of each other.
The proof of Lemma~\ref{lem:three} is complete.
\end{proof}

\begin{proof}[Proof of Lemma~\ref{lem:proper}] 
Suppose that $p,q,r,s$ are bounded:  
\begin{equation*}
|p|, \ |q|, \ |r|, \ |s| \  \le C
\end{equation*}
for some constant $C >0$.    
In this proof  $C$ denotes a constant that can vary from line to line, 
and depends on the the previous constants.   
By Lemma~\ref{lem:three},  three of $a,b,c,d$ are within bounded distance of each other.   
We assume that $a,b,c$  
are within bounded distance of each other.
The remaining three cases being identical to this one.
We must show that $\vert a\vert, \vert b\vert, \vert c\vert, \vert d\vert$
are within bounded distance of each other.



First suppose that $d$ is within bounded distance of $a$. 
Then 
$d$ is within bounded distance of $a,\  b$ and $c$.  
Assume all distances between any two of $a,b,c,d$ is $\le C$.   
Equations (\ref{eq:sums}) immediately bound at least $3$ of the quantities $|a + b||$, $|b+d|$, etc.  
Pick any one of them, say $|b + d|$. 
Then, since $|b -d |$ is bounded, $|b|$ is bounded.  
Since all pairwise distances are bounded, $|a|$, $|b|$, $|c|$, $|d|$ 
are all $\le C$, as desired.



Thus we may assume that $|a - d|$ is unbounded. 
Then so are $|b - d|$ and $|c - d|$.   
The proof that this cannot divides into two cases: 
\begin{itemize}
\item  $|a + d|$ is bounded. 
\item  $|a + d|$ is unbounded.  
\end{itemize}


\noindent
Suppose the first case, that is, when $|a + d|$ is bounded. 
Then $|b + d|$ and $|c+d|$ also remain bounded.  
Moreover,  $d$ must also be unbounded, since otherwise 
$$
|a - d| = |(a +d) - 2d|
$$ 
would be bounded.    Then the difference between the function 
\begin{equation*}
s(a,b,c,d) \ =\  4 - a^2 -b^2 - c^2 - abcd
\end{equation*}  
and the quartic polynomial $s(-d, -d, -d, d) $ in $d$ can be estimated by a polynomial of degree $3$ in $|d|$   
Thus 
\begin{equation*}
|s(a,b,c,d)| \ge C|d|^4.
\end{equation*} 
In particular $|s|$  is unbounded, contrary to hypothesis.


Finally, consider the  case
when $|a + d|$ is unbounded.  
Then $|b+d|$ and $|c+d|$ are also unbounded.  
Equations \eqref{eq:sums} bound all three quantities
$$
|a+b|, |b + c|, | a + c |.
$$  
Therefore $|a|$, $|b|$ and $| c|$ are all bounded, 
but $|d|$ is unbounded.   
So, for large $d$, the difference between the  polynomial $s(a,b,c,d)$
and the quadratic polynomial $s(0,0,0,d)$ can be estimated by a linear polynomial in $|d|$.   
Thus 
$$
|s (a,b,c,d)|\ge C|d|^2,
$$
hence $|s|$ is unbounded, again contrary to hypothesis.
The proof of Lemma~\ref{lem:proper} is complete.
\end{proof}

\begin{proof}[Conclusion of proof of Theorem~\ref{thm:onto}]
Since $\Phi$ is proper,
 its degree $\deg(\Phi)$ is defined.  
Since $\Phi$ is holomorphic,   $\deg(\Phi)\ge 0$.
Moreover $\deg(\Phi) > 0$ 
if and only the image of $\Phi$ contains an open set.
This occurs if and only if $\Phi$ is surjective.   
To check that the image contains an open set, use the inverse function theorem:  pick a point where the Jacobian determinant of $\Phi\ne 0$. 
For example, at the point 
$$
a=b-c=1, d=0,
$$ 
the Jacobian determinant is $-4\ne 0$.
\end{proof}

\subsection*{The more symmetric case}

\begin{proof}[Proof of Theorem~\ref{thm:torus}]
The finite group $\Delta$ of automorphisms of the tritangent plane
acts on $V$ by mapping 
\begin{align*}
p & \longmapsto \pm p \\ 
q & \longmapsto \pm q \\ 
r & \longmapsto \pm r  
\end{align*}
so that $V$ is $\Delta$-invariant $\Longleftrightarrow$ $p = q = r = 0$.
Then $V$ is of the form 
\begin{equation*}
x^2 +  y^2 +  z^2 +  x y z = s
\end{equation*}
which is equivalent under the linear change of variables
\begin{equation*}
\bmatrix x \\ y \\ z \endbmatrix  \longmapsto
\bmatrix -x \\ -y \\ - z \endbmatrix  
\end{equation*}
to the standard form of the $\slt$-character variety
of the 1-holed torus with boundary trace $s$. 
(Compare Goldman~\cite{Traces}.)
The proof of Theorem~\ref{thm:torus} is complete.
\end{proof}

%
%

\noindent
In terms of the boundary traces $a,b,c,d$ there are exactly two cases
in which this arises. 

\begin{thm}\label{thm:nolinearterms}
Let $p(a,b,c,d),q(a,b,c,d),r(a,b,c,d)$ 
be defined as in \eqref{eq:linearcoeffs}.
Then 
\begin{equation*}
p(a,b,c,d)=q(a,b,c,d)=r(a,b,c,d)=0  
\end{equation*}
if and only if (up to permutations of the variables
$a,b,c,d$) one of the two cases occurs:
\begin{itemize}
\item $a = b = c = 0$;
\item $a = b = c = -d$.
\end{itemize}
\end{thm}

\begin{proof}
We start with the following simple observation.
\begin{lemma}
Suppose 
$$
p(a,b,c,d)\ =\ q(a,b,c,d)\ =\ r(a,b,c,d)\ =\ 0.
$$ 
If one of $a,b,c,d$ vanishes, 
then at least three of $a,b,c,d$ vanish.
\end{lemma}
\begin{proof}
Suppose that $a=0$. Then:
\begin{align*}
p(a,b,c,d)=0 &\Longrightarrow cd =0, \\
q(a,b,c,d)=0&\Longrightarrow bc =0, \\
r(a,b,c,d)=0&\Longrightarrow bd =0
\end{align*}
Thus at least two of $b,c,d$ must also vanish.
\end{proof}

Thus to prove Theorem~\ref{thm:nolinearterms}, we may assume that
all $a,b,c,d$ are nonzero. Then
\begin{equation*}
\frac{a}{c} = - \frac{d}{b} = \frac{c}{a} 
\end{equation*}
since $p(a,b,c,d)=0$ and $q(a,b,c,d)=0$ respectively. 
Thus $a/c$ equals its reciprocal,
and hence equals $\pm 1$. 
Thus $b = \pm a$. 
Similarly $c = \pm a$ and $d = \pm a$. 
If $a=b=c=d$, or if $a,b,c,d$ fall into two equal pairs,
then one of $p,q,r$ will be nonzero. Thus three of $a,b,c,d$ are equal
and the other one equals the negative. The proof of
Theorem~\ref{thm:nolinearterms} is complete.
\end{proof}
\noindent
The first case, when $a = b = c = 0$, may be understood in terms of
the one-holed torus covering space of a disc with three branch points 
of order two. (Compare \cite{Erg} for details.)

\makeatletter \renewcommand{\@biblabel}[1]{\hfill#1.}\makeatother

\end{document}